\documentclass[10pt]{amsart}
\usepackage[colorlinks=true,linkcolor=red,citecolor=blue,urlcolor=blue,hypertexnames=false]{hyperref}
\usepackage{bookmark}
\usepackage{amsthm,thmtools,amssymb,amsmath,amscd}

\usepackage{fancyhdr}
\usepackage{esint}
\usepackage{enumerate}

\usepackage{pictexwd,dcpic}
\usepackage{footmisc}

\swapnumbers
\declaretheorem[name=Theorem,numberwithin=section]{thm}
\declaretheorem[name=Remark,style=remark,sibling=thm]{rem}
\declaretheorem[name=Lemma,sibling=thm]{lemma}

\numberwithin{equation}{section}

\usepackage{cleveref}
\crefname{lemma}{Lemma}{Lemmata}
\crefname{prop}{Proposition}{Propositions}
\crefname{thm}{Theorem}{Theorems}
\crefname{cor}{Corollary}{Corollaries}
\crefname{defn}{Definition}{Definitions}
\crefname{example}{Example}{Examples}
\crefname{rem}{Remark}{Remarks}
\crefname{assum}{Assumption}{Assumptions}
\crefname{notation}{Notation}{Notation}

\usepackage{autonum}

\newcommand{\ti}{\tilde}

\newcommand{\cn}{\colon}

\newcommand{\ov}{\overline}

\newcommand{\R}{\mathbb{R}}

\newcommand{\bbS}{\mathbb{S}}

\newcommand{\8}{\infty}

\newcommand{\be}{\beta}

\newcommand{\de}{\delta}
\newcommand{\e}{\epsilon}
\newcommand{\ka}{\kappa}

\newcommand{\D}{\Delta}
\newcommand{\G}{\Gamma}

\newcommand{\La}{\Lambda}

\newcommand{\cW}{\mathcal{W}}

\newcommand{\cR}{\mathcal{R}}

\newcommand{\del}{\partial}
\newcommand{\n}{\nabla}

\newcommand{\fa}{\forall}
\newcommand{\rt}{\sqrt}

\newcommand{\ip}[2]{\left\langle #1,#2 \right\rangle}
\newcommand{\fr}[2]{\frac{#1}{#2}}
\newcommand{\x}{\times}

\DeclareMathOperator{\tr}{tr}
\DeclareMathOperator{\Rm}{Rm}
\DeclareMathOperator{\Rc}{Rc}

\DeclareMathOperator{\grad}{grad}
\DeclareMathOperator{\ad}{ad}

\newcommand{\pf}[1]{\begin{proof} #1 \end{proof}}
\newcommand{\eq}[1]{\begin{equation}\begin{alignedat}{2} #1 \end{alignedat}\end{equation}}

\newcommand{\br}[1]{\left(#1\right)}
\newcommand{\enum}[1]{\begin{enumerate}[(i)] #1 \end{enumerate}}

\newcommand{\Ra}{\Rightarrow}
\newcommand{\ra}{\rightarrow}

\newcommand{\hp}{\hphantom}
\newcommand{\q}{\quad}

\newcommand{\auskommentieren}[1]{}

\newcommand{\beq}{\begin{equation}}
\newcommand{\eeq}{\end{equation}}
\newcommand{\bea}{\begin{equation}\begin{aligned}}
\newcommand{\eea}{\end{aligned}\end{equation}}

\usepackage{array} 
\usepackage{amssymb}
\usepackage{latexsym}
\usepackage{epic}
\usepackage{curves}
\usepackage{fancyhdr}
\usepackage{xspace}

\numberwithin{equation}{section}

\DeclareMathOperator{\grade}{grad}
\DeclareMathOperator{\Tr}{Tr}

\author[P. Bryan]{Paul Bryan}
\address{Department of Mathematics, Macquarie University
NSW 2109, Australia}
\email{paul.bryan@mq.edu.au}

\author[H. Kr\"oner]{Heiko Kr\"oner }
\address{Albert-Ludwigs-Universit\"{a}t,
Mathematisches Institut, Ernst-Zermelo-Str.~1, 79104
Freiburg, Germany}
\email{heiko.kroener@uni-freiburg.de}

\author[J. Scheuer]{Julian Scheuer }
\address{Albert-Ludwigs-Universit\"{a}t,
Mathematisches Institut, Ernst-Zermelo-Str.~1, 79104
Freiburg, Germany}
\email{julian.scheuer@math.uni-freiburg.de}\title[Harnack inequalities in Riemannian manifolds]
{Li-Yau gradient estimates for curvature flows in positively curved manifolds}
\begin{document}
\begin{abstract}
We prove differential Harnack inequalities for flows of strictly convex hypersurfaces by powers $p$, $0<p<1$,
of the mean curvature in Einstein manifolds with a positive lower bound on the sectional curvature. We assume that this lower bound is sufficiently large compared to the derivatives of the curvature tensor of the ambient 
space and that the mean curvature of the 
initial hypersurface is sufficiently large compared to the ambient geometry. We also obtain some new Harnack inequalities for more general curvature flows in the sphere, as well as a monotonicity estimate for the mean curvature flow in non-negatively curved, locally symmetric spaces.
\end{abstract}
\date{\today}

\maketitle
\section{Introduction}

In a seminal paper, Li and Yau \cite{LiYau:/1986} presented a new method to deduce a Harnack inequality for positive solutions of the heat equation
\eq{\del_{t}u-\D u=0}
on a compact Riemannian manifold of non-negative Ricci curvature, by proving the gradient estimate
\eq{\label{LY}\fr{\del_{t}u}{u}-\fr{|\n u|^{2}}{u^{2}}+\fr{n}{2t}\geq 0.}
Using an integration over space-time paths, from here it is possible to deduce a classical Harnack estimate, which gives an estimate between the spatial maxima and minima of the solution at different times. Hence \eqref{LY} is also called a {\it{differential Harnack estimate}}. After this work appeared, the study of differential Harnack estimates for curvature flows of convex hypersurfaces in Euclidean space was initiated at about the same time by Chow \cite{Chow:06/1991} for the Gauss curvature flow, Hamilton \cite{Hamilton:/1993,Hamilton:/1995} for the Ricci- and the mean curvature flow and by Andrews \cite{Andrews:09/1994} for more general curvature flows in the Euclidean space. It became apparent that the appropriate generalisation to positivity of solutions of the heat equation is some sort of positivity of curvature, such as convexity in the case of hypersurface flows. For example, for the mean curvature flow the differential Harnack estimate reads
\eq{\del_{t}H-b(\n H,\n H)+\fr{1}{2t}H\geq 0,}
where $b$ is the inverse of the second fundamental form, $H$ is its trace and $\n$ the Levi-Civita connection of the induced metric of the flow hypersurfaces. Here the condition of convexity ensures that $b$ is defined.

 Several other similar results followed, see \cite{Ivaki:11/2015,Li:/2011,Smoczyk:/1997,Wang:11/2007}. The first differential Harnack estimates for flows in non-Euclidean ambient space were proven by PB and Ivaki for the mean curvature flow in the sphere \cite{BryanIvaki:08/2015}, followed up by a generalisation of the speed function in this setting due to PB, Ivaki and JS \cite{BIS1}. The most general ambient spaces, in which the deduction of a differential Harnack estimate was possible so far, are locally symmetric Einstein spaces of non-negative sectional curvatures, cf. \cite{BIS4}. 
 
The object of this paper is to investigate to what extent these assumptions can be relaxed.

In \Cref{main_result1}, we prove Harnack inequalities for flows by powers, $0<p<1$,
of the mean curvature in Einstein manifolds where we do not assume
that the ambient space is locally symmetric
but instead a lower bound for its sectional curvatures and a relation between the mean curvature and the ambient geometry.

We also prove Harnack inequalities in the sphere, \Cref{Sphere}, for flows by powers $F^p$, $p>1$ of a convex,
$1$-homogeneous curvature function $F$. These are completely new, but with the added restriction that the solutions satisfy a certain pinching.

Finally, when the Einstein condition is dropped, we can still get a monotonicity estimate for the so called {\it{Harnack quadratic}} in \cref{main_result2}.

We introduce our results more precisely, after fixing the relevant notation. Let $n\geq 2$.
In an $(n+1)$-dimensional Riemannian manifold $(N, \bar g)$ with Levi-Civita connection $\bar\n$ and Riemannian curvature
\eq{\ov{\Rm}(X,Y)Z=\bar\n_{X}\bar\n_{Y}Z-\bar\n_{Y}\bar\n_{X}Z-\bar\n_{[X,Y]}Z}
 we consider curvature flows, i.e. time-dependent families of immersions
\eq{x\cn (0,T)\x M\rightarrow N}
of a closed, connected and orientable smooth manifold $M^n$, which satisfy
\begin{equation} \label{original_param}
 \dot {x} = -f\nu,
\end{equation}
where the flow hypersurfaces
\eq{M_{t}=x(t,M)}
with induced metric $g$ and Levi-Civita connection $\n$
will always assumed to be strictly convex, i.e. the second fundamental form
\eq{h=-\bar g(\bar \n^{2} x,\nu)}
is positive definite with a particular choice of a smooth normal field $\nu$.
The function $f\in C^{\8}(\G_{+})$ with
\eq{\G_{+}=\{\ka\in \R^{n}\cn \ka_{i}>0\q\fa 1\leq i\leq n\}}
 is in increasing dependence on the principal curvatures of the flow hypersurfaces and by well known methods may also be viewed as a function of the second fundamental form and the metric or as a function of the Weingarten operator $\cW$,
\eq{\label{f(g,h)}f=f(g,h),\q f=f(\cW),}
\cite{Andrews:/2007,Gerhardt:/2006,Scheuer:06/2018}. Due to the positivity of $h$, the twice contravariant tensor $b=(b^{ij})$ given by
\eq{b^{ik}h_{kj}=\de^{i}_{j}}
is well defined, where we have used a coordinate representation of the respective quantities in a local frame $(e_{i})$.

Our first main result provides the first estimate of this kind in non-locally symmetric ambient spaces for flows by small powers of the mean curvature.
\begin{thm} \label{main_result1}
Let $n\geq 2$ and $(N^{n+1},\bar g)$ be a Riemannian Einstein manifold with lower bound $c_1>0$ on the sectional curvatures. Let $0<p<1$. Then along any strictly convex solution
\eq{x\cn (0,T)\x M^n\ra N
}
to
\eq{\label{Main-B}\dot{x}=-H^p\nu,}
which satisfies
\eq{\label{Main-A}H\geq \fr{2np}{\min(1-p,2p)}\|\bar\n\ov{\Rm}\|,
}
the following Harnack inequality holds:
\begin{equation}
\partial_t H^p-b(\nabla H^p, \nabla H^p)+\frac{p}{(p+1)t}H^p \ge 0.
\end{equation}
\end{thm}

At a first glance, the assumption \eqref{Main-A} seems quite strong and rather technical. However, in general ambient spaces a differential Harnack inequality is unlikely to hold, since the property of a flow hypersurface to be of constant mean curvature is generally not preserved under the flow \eqref{Main-B}. Such a failure of preservation would violate comparison of the mean curvature at different space time points obtained by integrating the differential Harnack inequality over space-time paths. The interaction of the ambient curvature and the flow thus makes an assumption like \eqref{Main-A} essential. Also note that in order to obtain convergence results in general manifolds, similar assumptions had to be made in earlier works, e.g. \cite{Andrews:/1994a,Huisken:/1986}.

As to the sphere, originally treated in \cite{BryanIvaki:08/2015} for the mean curvature flow and in \cite{BIS1} for powers $0<p\leq 1$ of a convex curvature function, here we prove a first result of this kind for powers $p>1$ after assuming a certain kind of pinching condition. Furthermore we improve the result in \cite{BIS1} by obtaining a bonus term in case of a convex $F$ and $p=1$:

\begin{thm}\label{Sphere}
Let $N=\mathbb{S}^{n+1}$ and $F\in C^{\infty}(\Gamma_+)\cap C^0(\bar{\G}_+)$ be a strictly monotone, $1$-homogeneous and convex curvature function with $F(1,\dots,1)=n$.
\enum{
\item If
$1< p<\infty$ and if the flow is pinched in the sense that the flow hypersurfaces satisfy
\eq{\label{Sphere-A}\sum_{i=1}^{n-1}\fr{\del F}{\del\ka_{i}}(\ka_{1},\dots,\ka_{n})\ka_{i}\geq \fr{p-1}{p+1} \fr{\del F}{\del\ka_{n}}(\ka_{1},\dots,\ka_{n})\ka_{n}}
with ordered principal curvatures
\eq{\ka_{1}\leq \dots\leq\ka_{n},}
then along any strictly convex solution of 
\eq{\dot{x}=-F^p\nu}
the following Harnack inequality is valid:
 \eq{\del_t F^p-b(\n F^p,\n F^p)+\fr{p}{(p+1)t}F^p\geq 0.}
 \item Any strictly convex solution to
 \eq{\dot{x}=-F\nu}
 satisfies the Harnack inequality with bonus term
 \eq{\del_t F-b(\n F,\n F)-F(0,\dots,0,1)F+\fr{1}{2t}F\geq 0.}
 }
\end{thm}

\begin{rem}
The pinching condition \eqref{Sphere-A} is fulfilled for example when the flow hypersurfaces are pinched in the sense
\eq{\ka_{1}\geq \e\ka_{n},}
since in this case we have 
\eq{cg^{ij}\leq F^{ij}\leq Cg^{ij}}
for suitable $C>c>0$ and hence
\eq{\fr{\del F}{\del\ka_{i}}\geq \fr{c}{C}\fr{\del F}{\del\ka_{n}},}
leading to
\eq{\sum_{i=1}^{n-1}\fr{\del F}{\del\ka_{i}}\ka_{i}\geq (n-1)\fr{c}{C}\e\fr{\del F}{\del\ka_{n}}\ka_{n}.} Since $c/C\ra 1$ as $\e\ra 1$, condition \eqref{Sphere-A} can be satisfied in this situation.
However, \eqref{Sphere-A} also allows other conditions: For example suppose
\eq{F=\rt{n|A|^{2}},}
then
\eq{\label{A2}\fr{\del F}{\del\ka_{i}}=\fr{n\ka_{i}}{F}}
and hence
\eq{\sum_{i=1}^{n-1}\fr{\del F}{\del \ka_{i}}\ka_{i}\geq \fr{n\ka_{n-1}}{F}\ka_{n-1}\geq \e^{2}\fr{n\ka_{n}}{F}\ka_{n}=\e^{2}\fr{\del F}{\del\ka_{n}}\ka_{n},}
provided the solution is only $\e$-$(n-1)$-pinched, i.e.
\eq{\ka_{n-1}\geq \e\ka_{n}.}
With a suitable $\e$ condition, \eqref{Sphere-A} can be satisfied.
\end{rem}

In \cref{Ancient} we use this result to show that under the assumptions of \cref{Sphere} any ancient and strictly convex solution satisfying the pinching condition \eqref{Sphere-A} must a shrinking family of spheres.

Finally we obtain the following pointwise monotonicity estimate, if we drop the Einstein condition while still imposing local symmetry.

\begin{thm} \label{main_result2}
Let $n\geq 2$ and $(N^{n+1},\bar g)$ be a locally symmetric Riemannian manifold with non-negative sectional curvature. Then along any convex solution
of the mean curvature flow in $N$, for each $\xi\in M$
the quantity
\begin{equation}
u(\cdot,\xi)=\fr{\del_t H}{H}(\cdot,\xi)-\fr{1}{H}b(\n H(\cdot,\xi),\n H(\cdot,\xi))
\end{equation}
is non-decreasing.
\end{thm}

This estimate is not as strong as a classical differential Harnack inequality, because the lower bound one may deduce for $u$ depends on its initial value, while the classical Harnack is independent of any initial values. However, the statement in \cref{main_result2} gives a pointwise estimate which is stronger than the estimate one gets for $\min H$ from the standard evolution equation
\eq{\label{Ev-H}\del_t H=\D H+\|A\|^2H+\tr\br{\ov{\Rm}(\cdot,\nu,\nu,\cdot)H}.
}

To begin, in Section \ref{sec:Main-Proof} we use  the quite general framework from \cite{BIS4}
to state a key evolution equation derived therein and adapt some estimates. The sphere case is then treated
in \Cref{spherical} and the Einstein case in \Cref{Einstein}. In \Cref{ancient} we present some applications to ancient solutions. Lastly in \Cref{section_lower_bound} we derive
the monotonicity of the smallest principal curvature under the flow by adapting the
corresponding Euclidean case from \cite{Schulze:/2005} to the Riemannian manifold setting.

\section{Evolution of the Harnack quadratic}\label{sec:Main-Proof}

As explained in the introduction, we are interested in the derivation of Harnack inequalities for curvature flows of the form
\begin{equation}
\dot{x}=-f\nu,
\end{equation}
where $f=f(h^i_j)$ is a curvature function. In most of the previous works on this topic (except for \cite{Andrews:09/1994,BIS4}) for this purpose the evolution of the so-called {\it{Harnack quadratic}}
\begin{equation}
Q=\fr{1}{f}\br{\dot{f}-b^{ij}\n_i f\n_j f} 
\end{equation}
was studied. Directly studying $Q$ is a tedious project, since the evolution of $b^{ij}$ is complicated, especially in general backgrounds, and $\n$ is time dependent, cf. \cite{BIS1,Chow:06/1991,Hamilton:/1995,Li:/2011,Smoczyk:/1997,Wang:11/2007}. Inspired by Andrews' use of the Gauss map parametrisation of convex hypersurfaces in the Euclidean space \cite{Andrews:09/1994}, in \cite{BIS4} a tangential reparametrisation of the flow,  mimicking the behavior of the Gauss map, was used to circumvent these difficulties in general backgrounds. Namely, under the flow
\begin{equation}\label{BIS-Flow}
\dot{x}= -f\nu-x_{\ast}(\grad_h f),
\end{equation}
where
\eq{h(\grad_{h}f,X)=df(X)\q\fa X\in TM,}
it can easily be calculated that the normal is parallel,
\begin{equation}
\frac{\bar\n}{dt}\nu =0.
\end{equation}
This leads to the additional nice property, that the Harnack quadratic simplifies to $\del_t{\log f}$ in the new parametrisation, making computations a lot easier. In \cite{BIS4} this was the gateway to obtain the evolution of the Harnack quadratic in {\bf{general Riemannian and Lorentzian manifolds}}. Let us fix some more specific notation and conventions before we state this evolution equation.

We will usually omit the mapping $x_{*}$ when it is clear that a tangent vector $X\in TM$ has to be understood as its push-forward $x_{*}(X)$ to $TN$.
Define
\begin{equation}
u=\frac{\dot{f}}{f},
\end{equation}
compare \cite{Andrews:09/1994,BIS4,LiYau:/1986}. 
Let $\tilde \nabla$ denote the Levi-Civita connection of
\eq{\tilde g=\frac{h}{f}.}
Referring to \eqref{f(g,h)}, we will denote by $d_{\cW}f=(f^i_j)$ the derivative of $f$ with respect to $\cW$ and by $d_hf=(f^{ij})$ we denote the partial derivative of $f$ when understood as a function of the pair $(g,h)$. Note that 
\begin{equation}
f^{ij}=g^{ik}f_{k}^j
\end{equation}
in a local frame, compare for example \cite{Gerhardt:/2006} or \cite[Prop.~3.3]{Scheuer:06/2018}.

Define $A\in T^{1,1}(M)$ by 
\begin{equation}
 x_{*}(A(X))=-\bar \nabla_X\dot x,
\end{equation}
where we note that $\bar\n_X \dot{x}$ is tangential \cite[equ.~(3.2)]{BIS4},
and define the $T^{0,2}$-tensor $\Lambda$
\begin{equation}
 \Lambda(X,Y)=\overline{\Rm}(\dot x, x_{*}X, \nu, x_{*}Y)
\end{equation}
for $X,Y \in TM$ and its associated representation $\Lambda^{\sharp}$ as $T^{1,1}$-tensor by raising its last index with the 
metric $g$. Then the evolution of the Weingarten tensor $\cW$ is given by
\eq{\label{eq:dotW}\dot\cW=A\circ\cW+\La^{\sharp},}
cf. \cite[equ.~(3.6)]{BIS4}.
The evolution equation from \cite[Lemma 3.7]{BIS4}
for $u$ adapted to our setting can be stated as
\begin{equation} \label{main_equation_harnack}
\begin{aligned}
 Lu :=& \dot u -d_hf(\tilde \nabla^2 u)
 -d_hf(\tilde g(D_{(\cdot)}\grade_{\tilde g}u, \cdot))
 +\frac{1}{f}d_hf(\overline{\Rm}(\grade_{\tilde g}u, \cdot, \nu, \cdot))\\
 =& 
 \frac{1}{f}d_{\cW}^2f(\dot \cW, \dot \cW)+\frac{2}{f}
 d_hf ( h ( A ( \cdot ), A ( \cdot ) ) ) + \frac{2}{f}d_{\cW}f(A\circ\Lambda^{\sharp}-\Lambda^{\sharp}\circ A) \\
 &+ \left(1-\frac{d_{\cW}f(\cW)}{f}\right)
 \overline{\Rm}(\dot x, \nu, \nu, \dot x) + \frac{2}{f}d_hf(\overline{\Rm}(\cdot, \dot x, \dot x, \cW(\cdot))) \\
 &+ \frac{1}{f}d_hf(\bar \nabla \overline{\Rm}(\dot x, \cdot, \nu,\cdot, \dot x)+\bar \nabla \overline{\Rm}(\dot x, \cdot, \nu, 
 \dot x, \cdot)),
\end{aligned}
\end{equation}
where $D$ is the difference tensor
\eq{D_{X}Y=\hat\n_{X}Y-\ti\n_{X}Y}
and $\hat\n$ is defined via the decomposition
\eq{\bar\n_{X}Y=x_{\ast}(\hat\n_{X}Y)+\ti g(X,Y)\dot x,}
see \cite[p.~19]{BIS4}.

Following \cite{BIS4}  we want to estimate $u$ from below 
by using (\ref{main_equation_harnack}) and the parabolic maximum principle
in order to prove our main result. 
Therefore we need to deal with the first two terms on the right hand side of \eqref{main_equation_harnack}. We assume 
\begin{equation} 
f=F^p,\quad p>0,
\end{equation}
with a 1-homogeneous, strictly monotone and convex curvature function $F\in C^{\infty}(\Gamma_{+})$, where $F=H$ if $N$ is not a spaceform.

\begin{lemma}\label{Lu}
 Suppose that
 \begin{enumerate}
 \item[(i)] $F\in C^{\8}(\G_{+})$ is 1-homogeneous, strictly monotone, convex and satisfies $F(1,\dots,1)=n$, if $N$ is a spaceform and
 \item[(ii)] $F=H$, otherwise.
 \end{enumerate}
Then along any strictly convex solution to \eqref{BIS-Flow} with $f=F^p$, $0<p<\8$, we have
 \begin{equation}\label{Lu-1}
\begin{aligned}
 Lu \ge& \frac{p+1}{p}u^2-\frac{4}{p}\frac{d_{\cW}f(\Lambda^{\sharp})}{f}u + \frac{2}{p}\left(\frac{d_{\cW}f(\Lambda^{\sharp})}{f}\right)^2 \\
 	&+\frac{2}{f}d_{\cW}f(A\circ\Lambda^{\sharp}-\Lambda^{\sharp}\circ A)+ \left(1-\frac{d_{\cW}f(\cW)}{f}\right)
 \overline{\Rm}(\dot x, \nu, \nu, \dot x) \\
 	&+ \frac{2}{f}d_hf(\overline{\Rm}(\cdot, \dot x, \dot x, \cW(\cdot))) \\
    &+ \frac{1}{f}d_hf(\bar \nabla \overline{\Rm}(\dot x, \cdot, \nu, \cdot, \dot x)+\bar \nabla \overline{\Rm}(\dot x, \cdot, \nu, 
 \dot x, \cdot)). 
\end{aligned}
 \end{equation}
 \end{lemma}

\begin{proof} 
For completeness we present the proof from \cite{BIS4} with the necessary adaptions.
We prepare some auxiliary estimates for terms appearing on the right-hand side
of (\ref{main_equation_harnack}). Referring to \cite[Theorem 2.3]{BIS4}, there exists a linear map, denoted $f'(\cW)$ commuting with $\cW$ and such that for any $B$,
\[
d_{\cW} f (B) = \Tr (f'(\cW) \circ B).
\]
Then commuting \(f(\cW)\) with \(\cW\) and using that the trace is invariant under cyclic permutation we have
\[
\begin{split}
d_{\cW} f (\cW\circ A) &= \Tr (f'(\cW) \circ \cW \circ A) = \Tr (\cW \circ f'(\cW) \circ A) \\
&= \Tr (f'(\cW) \circ A \circ \cW) = d_{\cW} f (A \circ \cW).
\end{split}
\]
Using this, that \(\cW\) is self-adjoint so that \(\ad(\cW) = \cW\), the fact that (see \cite[Remark 2.10]{BIS4})
\[
d_{\cW}f(\ad(\cW\circ A)\circ \cW^{-1}\circ \cW\circ A) \ge \frac{1}{p}f^{-1}(d_{\cW}f(\cW\circ A))^2,
\]
and the evolution of \(\cW\) from equation \eqref{eq:dotW}, we estimate
\begin{equation}
\begin{aligned}
 d_hf(h(A(\cdot), A(\cdot))) =& 
 d_hf(g(A(\cdot), \cW \circ A(\cdot))) \\
=& d_hf(g(\cdot, \ad(A)\circ \cW \circ A(\cdot)))\\
 =& d_{\cW}f(\ad(A)\circ \cW\circ A)\\
 =& d_{\cW}f(\ad(\cW\circ A)\circ \cW^{-1}\circ \cW\circ A)\\
 \ge& \frac{1}{p}f^{-1}(d_{\cW}f(\cW\circ A))^2 \\
 =& \frac{1}{p}f^{-1}\br{d_{\cW}f(\Lambda^{\sharp}-\dot \cW)}^2 \\
 =& \frac{1}{pf}\br{\br{d_{\cW}f(\dot \cW)}^2-2d_{\cW}f(\dot \cW)d_{\cW}f(\Lambda^{\sharp})+\br{d_{\cW}f(\Lambda^{\sharp})}^2}.
\end{aligned}
 \end{equation}
For the second derivative term, in the case $F=H$ and \(f = H^p\), by \cite[(2.17)]{Scheuer:06/2018} we see that when \(p = 1\),
\begin{equation}
 d_{\cW}^2f(\dot \cW, \dot \cW)=0
 \end{equation}
For general \(p\) we have
\begin{equation}
 d_{\cW}f(\dot \cW) = pH^{p-1}\tr(\dot \cW)
\end{equation}
and 
\begin{equation}
\begin{aligned}
 d_{\cW}^2f(\dot \cW, \dot \cW)&=p(p-1)H^{p-2}\br{\tr(\dot \cW)}^2 \\
 &= \frac{p-1}{p}f^{-1}\br{d_{\cW}f(\dot \cW)}^2.
\end{aligned}
 \end{equation}
In the case that $F$ is convex and $f=F^p$ we obtain from \cite[Lemma~5.1]{BIS4} that
\begin{equation}
d_{\cW}^2f(\dot \cW,\dot \cW)\geq \frac{p-1}{p}f^{-1}\br{d_{\cW}f(\dot \cW)}^2
\end{equation}
In all cases, together with $\dot f= fu=d_{\cW}f(\dot \cW)$, we conclude
 \begin{equation}
  \begin{aligned}
   & \frac{2}{f}d_hf(h(A(\cdot), A(\cdot)))+\frac{1}{f}d_{\cW}^2f(\dot \cW, \dot \cW) \\
\ge& \frac{p+1}{p}f^{-2}\br{d_{\cW}f(\dot \cW)}^2-\frac{4}{p}f^{-2}d_{\cW}f(\dot \cW)d_{\cW}f(\Lambda^{\sharp})+\frac{2}{p}f^{-2}\br{d_{\cW}f(\Lambda^{\sharp})}^2\\
=& \frac{p+1}{p}u^2-\frac{4}{p}\frac{d_{\cW}f(\Lambda^{\sharp})}{f}u + \frac{2}{p}\left(\frac{d_{\cW}f(\Lambda^{\sharp})}{f}\right)^2.
   \end{aligned}
 \end{equation}
 Inserting this into \eqref{main_equation_harnack} gives the result.
\end{proof}

For convenience of notation, let
\begin{equation}
\label{eq:S}
S = \frac{d_{\cW}f(\Lambda^{\sharp})}{f}
\end{equation}
and
\begin{equation}
\label{eq:R}
\begin{aligned}
\mathcal{R} &= \frac{2}{f}d_{\cW}f(A\circ\Lambda^{\sharp}-\Lambda^{\sharp}\circ A)+ \left(1-\frac{d_{\cW}f(\cW)}{f}\right) \overline{\Rm}(\dot x, \nu, \nu, \dot x) \\
&+ \frac{2}{f}d_hf(\overline{\Rm}(\cdot, \dot x, \dot x, \cW(\cdot))) \\
&+ \frac{1}{f}d_hf(\bar \nabla \overline{\Rm}(\dot x, \cdot, \nu, \cdot, \dot x)+\bar \nabla \overline{\Rm}(\dot x, \cdot, \nu, \dot x, \cdot)).
\end{aligned}
\end{equation}

Using the evolution equation satisfied by $u$, we show general conditions, under which we can obtain a Harnack inequality. 

\begin{thm}\label{Master_Harnack}
Let $F\in C^{\8}(\G_{+})$ be 1-homogeneous, strictly monotone and convex with $F(1,\dots,1)=n$ if $N$ is a spaceform and $F=H$ otherwise. Let $x$ be a strictly convex solution to \eqref{BIS-Flow} with $f=F^{p}$, $0<p<\8$.
Let \(\beta\in\R\) satisfy
\begin{equation}
\begin{cases}
\beta &\leq \frac{2}{p + 1} S \\
\beta &\notin (\beta_-, \beta_+)
\end{cases}
\end{equation}
throughout $(0,T)\x M$,
where \(\beta_{\pm}(t,\xi)\) are the roots of
\begin{equation}
p(\beta) = \frac{p+1}{p} \beta^2 - \frac{4}{p} S \beta + \frac{2}{p} S^2 + \mathcal{R},
\end{equation}
and where we take it that the second condition, \(\beta \notin (\beta_-, \beta_+)\) is satisfied for all \(\beta\) in case there is at most one root.
Then
\begin{equation}
u - \beta + \frac{p}{p+1} \frac{1}{t} \geq 0.
\end{equation}
Equivalently, transforming back to the standard parametrisation,
\begin{equation}
\partial_t F^p - b(\nabla F^p, \nabla F^p) - \beta F^p + \frac{p}{p+1}\frac{1}{t} F^p \geq 0.
\end{equation}
\end{thm}

\begin{proof}
We write the estimate \eqref{Lu-1} with the help of \eqref{eq:S} and \eqref{eq:R} more compactly as
\begin{equation}
\label{eq:Lu}
Lu - \frac{p+1}{p}u^2 + \frac{4}{p} S u \geq \frac{2}{p} S^2 + \mathcal{R}.
\end{equation}
For any \(\beta \in \R\) consider the function
\begin{equation}
\eta(t) = -\frac{p}{p+1} \frac{1}{t} + \beta.
\end{equation}
We seek conditions on \(\beta\) so that
\begin{equation}
Lu - \frac{p+1}{p}u^2 + \frac{4}{p} S u \geq L\eta - \frac{p+1}{p}\eta^2 + \frac{4}{p} S \eta.
\end{equation}
By equation \eqref{eq:Lu}, it's sufficient that
\begin{equation}
L\eta - \frac{p+1}{p}\eta^2 + \frac{4}{p} S \eta \leq \frac{2}{p} S^2 + \mathcal{R}.
\end{equation}
Since \(\eta\) depends only on \(t\), we have \(L \eta = \partial_t \eta\) and hence
\begin{equation}
L \eta - \frac{p+1}{p}\eta^2 + \frac{4}{p} S \eta = 2\left(\beta - \frac{2}{p+1} S\right) \frac{1}{t} + \frac{4}{p} S \beta - \frac{p+1}{p} \beta^2.
\end{equation}
Then
\begin{equation}
L\eta - \frac{p+1}{p}\eta^2 + \frac{4}{p} S \eta \leq \frac{2}{p} S^2 + \mathcal{R}
\end{equation}
if and only if
\begin{equation}
2\left(\beta - \frac{2}{p+1} S\right) \frac{1}{t} + \frac{4}{p} S \beta - \frac{p+1}{p} \beta^2 \leq \frac{2}{p} S^2 + \mathcal{R}.
\end{equation}
That is,
\begin{equation}
2\left(\frac{2}{p+1} S - \beta\right) \frac{1}{t} + \frac{p+1}{p} \beta^2 - \frac{4}{p} S \beta + \frac{2}{p} S^2 + \mathcal{R} \geq 0.
\end{equation}
Since we require this inequality to hold for all \(t > 0\), the coefficient of the \(1/t\) term must be non-negative. For the remainder, at each $(t,\xi)\in (0,T)\x M$ let \(\beta_{\pm}\) denote the roots of
\begin{equation}
p(\beta) = \frac{p+1}{p} \beta^2 - \frac{4}{p} S \beta + \frac{2}{p} S^2 + \mathcal{R},
\end{equation}
given by
\eq{
\beta_{\pm}= \frac{2}{p+1} \left(S \pm \sqrt{\frac{1-p}{2} S^2 - \frac{p(p+1)}{4} \mathcal{R}}\right).
}
The conditions on \(\beta\) are then
\begin{equation}
\begin{cases}
\beta &\leq \frac{2S}{p + 1} \\
\beta &\notin (\beta_-, \beta_+)
\end{cases}
\end{equation}
where we take it that the second condition is satisfied for all \(\beta\) in case there is at most one root, since then \(p(\beta) \geq 0\) for all \(\beta\).

To summarise then, for \(\beta\) as in the hypothesis of the theorem we have
\begin{equation}
Lu - \frac{p+1}{p}u^2 + \frac{4}{p} S u \geq L\eta - \frac{p+1}{p}\eta^2 + \frac{4}{p} S \eta
\end{equation}
and \(\eta(0) = -\infty\). Then the maximum principle implies \(u \geq \eta\) for all \(t > 0\) which is precisely the required conclusion.
\end{proof}

\begin{rem}
\label{rem:harnack}
Note that if \(S\) and \(\be_{-}\) are bounded below, then we may take
\begin{equation}
\beta \leq \min\left\{\frac{2S}{p+1}, \beta_-\right\}
\end{equation}
and obtain a Harnack inequality. However, in this case \(\beta\) may be negative and the Harnack is weaker than typically obtained. 

In the case that \(\mathcal{R} \geq 0\), we may discard that term and be a little more explicit replacing \(\beta_{\pm}\) with the roots of \(\frac{p+1}{p} \beta^2 - \frac{4}{p} S \beta + \frac{2}{p} S^2\):
\begin{equation}
\beta_{\pm} = \frac{2}{p+1} \left(S \pm \sqrt{\frac{1-p}{2}}|S|\right).
\end{equation}
In particular, if \(p > 1\), the roots are complex and any 
\eq{\beta \leq \tfrac{2S}{p+1}} suffices. 
If \(0 < p \leq 1\), then 
\eq{\beta_- \leq \tfrac{2S}{p+1} \leq \beta_+}
 and so the only option is to take \(\beta \leq \beta_-\). If \(S \leq 0\), then \(\beta_- \leq 0\) and we get a weak Harnack. If on the other hand, \(S \geq 0\), then
\begin{equation}
0 \leq \beta_- = \frac{2}{p+1}\left(1 - \sqrt{\frac{1-p}{2}}\right) S \leq \frac{2S}{p+1}
\end{equation}
and the Harnack includes a good non-negative \emph{bonus} term that is strictly positive when \(S \geq c> 0\).
\end{rem}

\section{The spherical case} \label{spherical}

We use \cref{Master_Harnack} and \cref{rem:harnack} to prove \Cref{Sphere}.

On the sphere,
\[
\ov{\Rm} (X, Y, Z, W) = \ov{g}(X, W) \ov{g} (Y, Z) - \ov{g}(X, Z)\ov{g} (Y, W)
\]
and hence
\eq{S=\fr{d_{\cW}f(\La^{\sharp})}{f}=\fr{p}{F}F^{ij}\ov{\Rm}(\dot x,\n_{i}x,\nu,\n_{j}x)=-\fr{p}{F}F^{ij}\bar g(\dot x,\nu)g_{ij}=pF^{p-1}F^{ij}g_{ij}\geq 0,}
According to \cref{rem:harnack}, to prove \Cref{Sphere} part (i) it then suffices to show that $\mathcal{R}\geq 0$.

Recalling the Euler identity for degree 1 homogeneous functions,
\[
F^{ik} h^j_k g_{ij} = F^{ik}h_{ik} = F
\]
we have
\eq{\cR&=(1-p)\ov{\Rm}(\dot x,\nu,\nu,\dot x)+\fr{2p}{F}F^{ik}h_k^j\ov{\Rm}(\n_{i}x,\dot x,\dot x,\n_{j}x)\\
&=(1-p)(\|\dot x\|^2-\ip{\dot x}{\nu}^2)+\fr{2p}{F}F^{ik}h^j_k\br{\|\dot x\|^2g_{ij}-\ip{\dot x}{\n_{i}x}\ip{\dot x}{\n_{j}x}}\\
&\geq  (1+p)(\|\dot x\|^2-\ip{\dot x}{\nu}^2)-\fr{2p}{F}F^{ik}h^j_kb^{ml}\n_{m}fg_{li}b^{rs}\n_{r}fg_{sj}\\
&=(1+p)b^l_jb^{jk}\n_{k}f\n_{l}f-\fr{2p}{F}F^{ik}b^{m}_i\n_{m}f\n_{k}f\\
&=\fr{1}{F}\br{(1+p)Fb^{jk}-2pF^{jk}}b^l_j\n_{k}f\n_{l}f
}
We use the assumed pinching condition to estimate
\begin{equation}\begin{aligned}\label{Sphere-Proof-2}
(1+p)\frac{F}{\kappa_j}-2p\fr{\del F}{\del \ka_{j}}&\geq (1+p)\frac{F}{\kappa_n}-2p\fr{\del F}{\del\ka_{j}}\\
		&=(1+p)\fr{\del F}{\del\ka_{n}}+(1+p)\sum_{i=1}^{n-1}\fr{\del F}{\del\ka_{i}}\fr{\ka_{i}}{\ka_{n}}-2p\fr{\del F}{\del\ka_{j}}\\
		&\geq\br{1+p}\br{1+\fr{p-1}{p+1}}\fr{\del F}{\del \ka_{n}}-2p\fr{\del F}{\del\ka_{j}}\\
		&\geq 0
\end{aligned}
\end{equation}
due to the convexity of $F$, which implies
\eq{\fr{\del F}{\del \ka_1}\leq\dots\leq \fr{\del F}{\del\ka_n}\leq n.
}
Hence $\mathcal{R}\geq 0$, given the proposed pinching. In view of \Cref{rem:harnack}, \Cref{Sphere}(i) follows.

For part (ii), we have \(p=1\) and we may improve the non-negativity of $S$ to the following positive lower bound:
\eq{S=F^{ij}g_{ij}\geq \ka_{n}^{-1}F^{ij}h_{ij}=F\br{\fr{\ka_{1}}{\ka_{n}},\dots,1}\geq F(0,\dots,0,1). }
Hence the choice
\eq{\be=F(0,\dots,0,1)}
yields a Harnack estimate with bonus term and the proof is complete.

\section{More general ambient spaces}\label{Einstein}
\subsection*{Proof of {\texorpdfstring{\cref{main_result1}}{}}}
Again, to prove the Harnack inequality, We use \cref{Master_Harnack} and \cref{rem:harnack}.

For Einstein manifolds, we have
\[
\ov{\Rc} = \frac{\ov{R}}{n+1} \ov{g}.
\]
We calculate
\eq{S=\fr{d_{\cW}f(\La^\sharp)}{f}=\fr{p}{H}\tr\br{\ov{\Rm}(\dot{x},\cdot,\nu,\cdot)}=-\fr{p}{H}\ov{\Rc}(\dot x,\nu)=pH^{p-1}\fr{\bar R}{n+1}
}
and 
\eq{\cR&=(1-p)\ov{\Rm}(\dot x,\nu,\nu,\dot x)+\fr{2p}{H}h^{ij}\ov{\Rm}(\n_{i}x,\dot x,\dot x,\n_{j}x)\\
    &\hp{=}+\fr{p}{H}g^{ij}\br{{\bar{\n}} {\ov{\Rm}}(\dot x,\n_{i}x,\nu,\n_{j}x,\dot x)+{\bar{\n}} {\ov{\Rm}}(\dot x,\n_{i}x,\nu,\dot x,\n_{j}x)
    }
}

Writing 
\eq{\dot x=-f\nu-V}
and using the positive lower bound $c_1$ on the sectional curvatures we estimate
\eq{\ov{\Rm}(\dot x,\nu,\nu,\dot x)\geq c_1(\|\dot x\|^2-f^2)=c_1\|V\|^2
}
and in normal coordinates
\eq{h^{ij}\ov{\Rm}(\n_{i}x,\dot x,\dot x,\n_{j}x)&=\sum_{i=1}^n\ka_i \ov{\Rm}(\n_{i}x,\dot x,\dot x,\n_{i}x)\\
        &\geq c_1\sum_{i=1}^n\ka_i(\|\dot x\|^2-\ip{\n_{i}x}{\dot{x}}^2)\\
        &\geq c_1Hf^2.   
}

Hence 
\eq{\cR&\geq (1-p)c_1 \|V\|^2+2pc_1 f^2-\fr{2np}{H}\|\bar\n\ov{\Rm}\|\|\dot x\|^2\\
        &\geq \br{\min(1-p,2p)c_1-\fr{2np}{H}\|\bar\n\ov{\Rm}\|}\|\dot x\|^2\\
        &\geq 0
}
under the present assumptions. Due to $p<1$, $S$ is not uniformly under control from below and we complete the proof of \cref{main_result1} by choosing $\beta=0$. 

\subsection*{Proof of {\texorpdfstring{\cref{main_result2}}{}}}
We use \cref{Lu} to prove that 
\eq{u=\fr{\del_t H}{H}}
is bounded from below by its initial values in the parametrisation
\eq{\dot{x}=-H\nu-V.}
Reverting to the standard parametrisation then gives the result.
\eqref{Lu-1} gives
\eq{Lu\geq 2\br{u-\fr{d_{\cW}f(\La^{\sharp})}{f}}^2+\fr{2}{H}h^{ij}\ov{\Rm}(\n_{i}x,\dot x,\dot x,\n_{j}x)\geq 0.} The result follows using the maximum principle.

\section{Ancient solutions to curvature flows in the sphere} \label{ancient}

A solution $x$ to a curvature flow equation is called {\it{ancient}}, if it is defined on $(-\8,0)\x M$, i.e. for all times of its existence, it has been existing forever.

Convex ancient solutions for flows in the Euclidean space arise as singularity models after a suitable space-time blow up around some singularity. Hence it is important and has been a widely studied field to find conditions which allow to classify convex ancient solutions. In the Euclidean space there is a full classification fur the curve shortening flow as shrinking spheres or so called {\it{Angenent ovals}}, cf. \cite{DaskalopoulosHamiltonSesum:/2010}, while for the mean curvature flow in higher dimensions there are various equivalent characterisations of when a convex ancient solution must be a family of shrinking spheres, e.g. \cite{HaslhoferHershkovits:/2016,HuiskenSinestrari:10/2015,Langford:08/2017}. Maybe the simplest characterisation is in terms of a uniform pinching of the form
\eq{\ka_n\leq c\ka_1.}
For flows in the sphere the situation is much more rigid, and such a strong pinching condition is unnecessary for a wide range of flows. The study of ancient solutions to curvature flows in the sphere was initiated by PB and Louie in \cite{BryanLouie:04/2016} for the curve shortening flow and in \cite{BryanIvaki:08/2015,HuiskenSinestrari:10/2015} for the mean curvature flow. The outcome in these papers were that all convex ancient solutions are shrinking spheres. There is a deep geometric reason behind this result: Roughly, a sufficiently regular convex hypersurface of the sphere can not be too large to fit into an open hemisphere without being an equator, cf. \cite{MakowskiScheuer:11/2016}. So if the backwards solution has a uniform curvature bound (or equivalently satisfies a uniform interior ball condition with a uniformly positive radius), the backwards limit already must be an equator. With this at hand, it is a standard Alexandrov reflection argument to show that the symmetry of the equator carries over to the solution. In the paper \cite{BIS2} this argument is made rigorous in order to classify ancient solutions for a wide range of speed functions, a range that by far exceeds the class typically considered in such problems. The only assumption on the flow is due to the above argument: One needs a uniform interior ball condition backwards in time to conclude the backwards limit to be an equator. Therefore a bound on the mean curvature is assumed.

In many particular situations, for example when there is a Li-Yau estimate available, such a backwards bound can be provided and hence a full classification follows. In this section we apply this method for the flows in $\bbS^{n+1}$ we have considered in \cref{spherical} and obtain the following application of our Harnack inequality.

\begin{thm}\label{Ancient}
Let $N=\mathbb{S}^{n+1}$ and $F\in C^{\infty}(\Gamma_+)\cap C^0(\bar{\G}_+)$ be a  strictly monotone, $1$-homogeneous and convex curvature function and let
$1< p<\infty$.
If a strictly convex, ancient solution to
\eq{\dot{x}=-F^p\nu}
is pinched in the sense that
\eq{\sum_{i=1}^{n-1}\fr{\del F}{\del\ka_{i}}(\ka_{1},\dots,\ka_{n})\ka_{i}\geq \fr{p-1}{p+1} \fr{\del F}{\del\ka_{n}}(\ka_{1},\dots,\ka_{n})\ka_{n}}
with ordered principal curvatures
\eq{\ka_{1}\leq \dots\leq\ka_{n},}
it is a family of shrinking spheres.
\end{thm}

\begin{rem}
Note that when looking at speeds which are not homogeneous of degree one, the anomaly can occur that there are no non-equatorial ancient solutions at all. For example, when $p<1$, the shrinking sphere solution to the power mean curvature flow
\eq{\label{Hp}\dot{x}=-H^p\nu}
only needs a finite time $T_S$ from the equator to a point. In this case ancient solutions are thus trivially classified and thus in \cite{BIS2} we introduced the notion of quasi-ancient solutions to \eqref{Hp}, which are solutions that exist on the interval $(-T_S,0)$ and classified these under the discussed assumptions.

For the case $p>1$ however, the next lemma shows the shrinking sphere solution is ancient and hence the formulation of \cref{Ancient} is justified. The proof is a simple ODE comparison argument and is omitted.
\end{rem}

\begin{lemma}
Let $F\in C^{\8}(\G_{+})$ be a strictly monotone and $1$-homogeneous curvature function. For $p>1$ any shrinking sphere solution to
\eq{\dot x=-F^{p}\nu} exists on \((-\infty, 0)\) with limit an equator as \(t \to -\infty\) and collapsing to the central point as \(t \to 0\).
\end{lemma}

\subsection*{Proof of {\texorpdfstring{\cref{Ancient}}{}}}
Under the present assumptions we know that for every $T>0$ and at each $(t,\xi)\in (-T,0)\x M$ the Harnack inequality
 \eq{\del_t F^p-b(\n F^p,\n F^p)+\fr{p}{(p+1)(t-T)}F^p\geq 0}
 is valid and hence, fixing a point $(t,\xi)$ while letting $T\ra \-\8$, we obtain
 \eq{\del_t F^p (t,\xi)\geq 0\q\fa (t,\xi)\in (-\8,0)\x M.}
 Hence $F$ is bounded backwards in time. Due to the convexity of $F$ we have
 \eq{H\leq F,}
 compare \cite[Lemma~2.2.20]{Gerhardt:/2006}. Hence the assumptions of \cite[Thm.~1.2]{BIS2} are satisfied and we conclude that the solution must be a family of shrinking spheres.

\appendix

\section{Lower bound on the smallest principal curvature}\label{section_lower_bound}

We deduced the estimate in \cref{main_result1} under the assumption that the mean curvature is large enough and the hypersurfaces are convex. In order to apply this result forward in time, it is of interest whether these properties are preserved along the flow. From the evolution of $H^p$, which looks similar to \eqref{Ev-H} for the case $p=1$, it is clear that the minimum of the mean curvature is non-decreasing. Where convexity is concerned, in this appendix we prove that the smallest principal curvature remains above a certain threshold if this is the case initially. We prove this in a setting more general than in \cref{main_result1}.

\begin{thm}
Let $n\geq 2$ and $(N^{n+1},\bar g)$ be a Riemannian manifold with bounded geometry and $0<p<1$. Let $x$ be a solution to
\eq{\dot{x}=-H^p \nu.}
Then there exists a constant $c>0$, such that
\eq{\min_{M} \ka_1(0,\cdot)\geq c\q\Ra\q \min_{M}\ka_1(t,\cdot)\geq \fr cn\q\fa t>0.}
\end{thm}

\pf{
It suffices to prove that 
\eq{\ti H=\tr(b)=\sum_{i=1}^n\fr{1}{\ka_i}}
remains small, if it is small initially. The evolution equation of $b$ is related to the one of $\cW$ via
\eq{\dot{b}^r_s-pH^{p-1}\D b^r_s=-b^r_i b^j_s\br{\dot{h}^i_j-pH^{p-1}\D h^i_j}-2pH^{p-1}h^l_m \n_{k}b^r_{l}  \n^{k}{b^m_s}. 
}
Now we can use the evolution of the Weingarten operator for flows in general Riemannian spaces, cf. \cite[Lemma~2.4.1]{Gerhardt:/1996}, to deduce
\eq{\label{App-1}\del_t\ti H&\leq pH^{p-1}\D\ti H-pH^{p-1}\|\cW\|^2\ti H+n(p-1)H^p\\
            &\hp{=}-p(p-1)H^{p-2}\n^{i}H\n_{j}Hb^r_ib^j_r+cH^p\|\ov{\Rm}\|\ti H^2+cH^{p-1}\|\bar\n\ov{\Rm}\|\ti H^2\\
            &\hp{=}-2pH^{p-1}h^l_m \n_{k}b^r_{l}  \n^{k}{b^m_r}. 
}

There holds in normal coordinates
\eq{h^l_m\n_{k}b^r_{l}\n^{k}{b^m_{r}}=\ka_m\n_{k}b^r_{m}\n^{k}{b^m_{r}}\geq \ti H^{-1}\|\n b\|^2.
}

Due to the Codazzi equation there holds in normal coordinates
\eq{\n^{i}{H}\n_{j}Hb^r_ib^j_r&\leq {\n_{l}h^i_{k}}g^{kl}\n_{m}h^m_{j}b^r_ib^j_r+\ov{\Rm}\star\n H\star b^2+\ov{\Rm}\star \ov{\Rm}\star b^2\\
                    &=h^i_kh^m_jg^{kl}\n_{l}b^r_{i}\n_{m}b^j_{r}+\ov{\Rm}\star \n b\star b^2\star b^2+\ov{\Rm}\star \ov{\Rm}\star b^2\\
                    &=\ka_i\ka_j \n_{i}{b^{ri}}\n_{j}b^j_{r}+\ov{\Rm}\star \n b\star b^2\star b^2+\ov{\Rm}\star \ov{\Rm}\star b^2.}
Together we obtain
\eq{&p(1-p)H^{p-2}\n^{i}{H}\n_{j}Hb^r_ib^j_r-2pH^{p-1}h^l_m \n_{k}b^r_{l}  \n^{k}{b^m_r}\\
    \leq~& p(1-p)H^{p-2}\br{\ka_i\ka_j \n_{i}{b^{ri}}\n_{j}b^j_{r}-\fr{2}{1-p}H\ka_m\n_{k}b^r_{m}\n^{k}{b^m_{r}}}\\
    \hp{\leq~}+&cH^{p-2}\|\ov{\Rm}\|\|\n b\|\ti H^4+cH^{p-2}\|\ov{\Rm}\|^2\ti H^2\\
    \leq~&p(1-p)H^{p-2}\br{\sum_{i,j,r}\ka_i \ka_j \n_{i}b_{ri}^2-\fr{2}{1-p}H\sum_{i,j,r}\ka_i \n_{j}b_{ri}\n_{j}b_{ri}   }\\
    \hp{\leq~}+&cH^{p-2}\|\ov{\Rm}\|\|\n b\|\ti H^4+cH^{p-2}\|\ov{\Rm}\|^2\ti H^2\\
    \leq~&p(p-1)H^{p-1}\sum_{i,j,r}\ka_i \n_{j}b_{ri}\n_{j}b_{ri}   +cH^{p-2}\|\ov{\Rm}\|\|\n b\|\ti H^4+cH^{p-2}\|\ov{\Rm}\|^2\ti H^2\\
    \leq~&p(p-1)H^{p-1}\ti H^{-1}\|\n b\|^2+cH^{p-2}\|\ov{\Rm}\|\|\n b\|\ti H^4+cH^{p-2}\|\ov{\Rm}\|^2\ti H^2.
}
Hence using Cauchy-Schwarz, every term can be absorbed into either the good term involving $\|\n b\|^2$ or into the negative terms of \eqref{App-1}, provided $\ti H$ is small enough. This completes the proof.
}

\subsection*{Acknowledgements}
This work has been supported by the ``Deutsche Forschungsgemeinschaft'' (DFG, German research foundation) within the research grant ``Harnack inequalities for curvature flows and applications'', number SCHE 1879/1-1 and by the ARC within the research grant ``Analysis of fully non-linear geometric problems and differential equations'', number DE180100110.

\bibliographystyle{amsplain}
\bibliography{Bibliography}

\providecommand{\bysame}{\leavevmode\hbox to3em{\hrulefill}\thinspace}
\providecommand{\MR}{\relax\ifhmode\unskip\space\fi MR }
\providecommand{\MRhref}[2]{%
  \href{http://www.ams.org/mathscinet-getitem?mr=#1}{#2}
}
\providecommand{\href}[2]{#2}
\begin{thebibliography}{10}

\bibitem{Andrews:/1994a}
Ben Andrews, \emph{Contraction of convex hypersurfaces in {R}iemannian spaces},
  J. Differ. Geom. \textbf{39} (1994), no.~2, 407--431.

\bibitem{Andrews:09/1994}
\bysame, \emph{Harnack inequalities for evolving hypersurfaces}, Math. Z.
  \textbf{217} (1994), no.~1, 179--197.

\bibitem{Andrews:/2007}
\bysame, \emph{Pinching estimates and motion of hypersurfaces by curvature
  functions}, J. Reine Angew. Math. \textbf{608} (2007), 17--33.

\bibitem{BryanIvaki:08/2015}
Paul Bryan and Mohammad~N. Ivaki, \emph{Harnack estimate for mean curvature
  flow on the sphere}, preprint,
  {\href{https://arxiv.org/abs/1508.02821}{arxiv:1508.02821}}, 2015.

\bibitem{BIS2}
Paul Bryan, Mohammad~N. Ivaki, and Julian Scheuer, \emph{On the classification
  of ancient solutions to curvature flows on the sphere}, preprint,
  {\href{http://arxiv.org/abs/1604.01694}{arxiv:1604.01694}}, 2016.

\bibitem{BIS4}
\bysame, \emph{Harnack inequalities for curvature flows in {R}iemannian and
  {L}orentzian manifolds}, preprint,
  {\href{https://arxiv.org/abs/1703.07493}{arxiv:1703.07493}}, 2017.

\bibitem{BIS1}
\bysame, \emph{Harnack inequalities for evolving hypersurfaces on the sphere},
  Commun. Anal. Geom. \textbf{26} (2018), no.~5, 1047--1077.

\bibitem{BryanLouie:04/2016}
Paul Bryan and Janelle Louie, \emph{Classification of convex ancient solutions
  to curve shortening flow on the sphere}, J. Geom. Anal. \textbf{26} (2016),
  no.~2, 858--872.

\bibitem{Chow:06/1991}
Bennett Chow, \emph{On {H}arnack's inequality and entropy for the {G}aussian
  curvature flow}, Commun. Pure Appl. Math. \textbf{44} (1991), no.~4,
  469--483.

\bibitem{DaskalopoulosHamiltonSesum:/2010}
Panagiota Daskalopoulos, Richard Hamilton, and Natasa Sesum,
  \emph{Classification of compact ancient solutions to the curve shortening
  flow}, J. Differ. Geom. \textbf{84} (2010), no.~3, 455--464.

\bibitem{Gerhardt:/1996}
Claus Gerhardt, \emph{Closed {W}eingarten hypersurfaces in {R}iemannian
  manifolds}, J. Differ. Geom. \textbf{43} (1996), no.~3, 612--641.

\bibitem{Gerhardt:/2006}
\bysame, \emph{Curvature problems}, Series in Geometry and Topology, vol.~39,
  International Press of Boston Inc., Sommerville, 2006.

\bibitem{Hamilton:/1993}
Richard Hamilton, \emph{The {H}arnack estimate for the {R}icci flow}, J.
  Differ. Geom. \textbf{37} (1993), no.~1, 225--243.

\bibitem{Hamilton:/1995}
\bysame, \emph{{H}arnack estimate for the mean curvature flow}, J. Differ.
  Geom. \textbf{41} (1995), no.~1, 215--226.

\bibitem{HaslhoferHershkovits:/2016}
Robert Haslhofer and Or~Hershkovits, \emph{Ancient solutions of the mean
  curvature flow}, Commun. Anal. Geom. \textbf{24} (2016), no.~3, 593--604.

\bibitem{Huisken:/1986}
Gerhard Huisken, \emph{Contracting convex hypersurfaces in {R}iemannian
  manifolds by their mean curvature}, Invent. Math. \textbf{84} (1986), no.~3,
  463--480.

\bibitem{HuiskenSinestrari:10/2015}
Gerhard Huisken and Carlo Sinestrari, \emph{Convex ancient solutions of the
  mean curvature flow}, J. Differ. Geom. \textbf{101} (2015), no.~2, 267--287.

\bibitem{Ivaki:11/2015}
Mohammad~N. Ivaki, \emph{Centro-affine normal flows on curves: Harnack
  estimates and ancient solutions}, Ann. I. H. Poincare non linear Anal.
  \textbf{32} (2015), no.~6, 1189--1197.

\bibitem{Langford:08/2017}
Mat Langford, \emph{A general pinching principle for mean curvature flow and
  applications}, Calc. Var. Partial Differ. Equ. \textbf{56} (2017), no.~4,
  107.

\bibitem{LiYau:/1986}
Peter Li and Shing-Tung Yau, \emph{On the parabolic kernel of the
  {S}chr{\"o}dinger operator}, Acta Math. \textbf{156} (1986), no.~1, 153--201.

\bibitem{Li:/2011}
Yi~Li, \emph{Harnack inequality for the negative power {G}aussian curvature
  flow}, Proc. Am. Math. Soc. \textbf{139} (2011), no.~10, 3707--3717.

\bibitem{MakowskiScheuer:11/2016}
Matthias Makowski and Julian Scheuer, \emph{Rigidity results, inverse curvature
  flows and {A}lexandrov-{F}enchel type inequalities in the sphere}, Asian J.
  Math. \textbf{20} (2016), no.~5, 869--892.

\bibitem{Scheuer:06/2018}
Julian Scheuer, \emph{Isotropic functions revisited}, Arch. Math. \textbf{110}
  (2018), no.~6, 591--604.

\bibitem{Schulze:/2005}
Felix Schulze, \emph{Evolution of convex hypersurfaces by powers of the mean
  curvature}, Math. Z. \textbf{251} (2005), no.~4, 721--733.

\bibitem{Smoczyk:/1997}
Knut Smoczyk, \emph{Harnack inequalities for curvature flows depending on mean
  curvature}, N. Y. J. Math. \textbf{3} (1997), 103--118.

\bibitem{Wang:11/2007}
Jie Wang, \emph{Harnack estimate for the ${H}^k$ flow}, Sci. China Ser. A
  \textbf{50} (2007), no.~11, 1642--1650.

\end{thebibliography}

\end{document}